\documentclass[12pt]{article}

\usepackage{mathrsfs}
\usepackage[utf8x]{inputenc}
\usepackage{amsfonts}
\usepackage{amssymb}
\usepackage{amsmath}
\usepackage{amsthm}
\usepackage{lipsum}
\usepackage{graphicx}
\usepackage{float}
\usepackage{mathrsfs}
\usepackage{color}
\usepackage{bbm}
\usepackage[all]{xy}
\usepackage[mathscr]{eucal}
\usepackage{afterpage}
\usepackage{relsize}
\usepackage{mathtools}
\usepackage{enumitem}
\usepackage{tikz}
\usepackage{dsfont}
\usepackage{multirow, array}

\usepackage[top=1.2in, bottom=1.2in, left=1.2in, right=1.2in]{geometry}

\usepackage{enumitem}
\usepackage{color}

\title{De Rham's theorem for Orlicz cohomology in the case of Lie groups}
\author{Emiliano Sequeira}
\date{}

\usepackage{lipsum}

\newcommand{\R}{\mathbb{R}}

\newcommand{\Z}{\mathbb{Z}}
\newcommand{\N}{\mathbb{N}}

\newcommand{\Lie}{\mathrm{Lie}}

\newcommand{\supp}{\operatorname{supp}}
\newcommand{\Vol}{\operatorname{Vol}}
\newcommand{\Ker}{\operatorname{Ker}}
\newcommand{\I}{\operatorname{Im}}

\newtheorem{theorem}{Theorem}[section]
\newtheorem{proposition}[theorem]{Proposition}
\newtheorem{lemma}[theorem]{Lemma}
\newtheorem{corollary}[theorem]{Corollary}

\theoremstyle{definition}

\newtheorem{remark}[theorem]{Remark}
\newtheorem{example}[theorem]{Example}

\newcommand{\Addresses}{{
\bigskip
\footnotesize

\emph{Centro de Matem\'atica, Universidad de la Rep\'ublica, Uruguay};\par\nopagebreak 

\textit{E-mail address}, \texttt{esequeira@cmat.edu.uy}
}}

\begin{document}

\maketitle

\begin{abstract}
We prove the equivalence between the simplicial Orlicz cohomology and the Orlicz-de Rham cohomology in the case of Lie groups. Since the first one is a quasi-isometry invariant for uniformly contractible simplicial complexes with bounded geometry, we obtain the invariance of the second one in the case of contractible Lie groups. We also define the Orlicz cohomology of a Gromov-hyperbolic space relative to a point on its boundary at infinity, for which the same results are true.
\end{abstract}

\section{Introduction}

\subsection{Motivation}

Orlicz cohomology has been studied in recent works (\cite{KP,C,K17,GK19,BFP}) as a natural generalization of $L^p$-cohomology. These kind of cohomology theories can be defined in different contexts, which are related by some equivalence theorems . A motivation to study them is that they provide quasi-isometry invariants, so they can be applied to quasi-isometry classification problems (see for example \cite{Pa08,C}).

%(\cite{GKS3,Pa88,GKS4,G93,Pa95,E,BP,CT,G,C,BR})

For $L^p$-cohomology we have a de Rham-type theorem, which establishes an equivalence between its simplicial version and its de Rham version (see \cite{GKS4,Pa95,G}). This equivalence is important to work with $L^p$-cohomology because it allows to use one or the other as appropiated. For example, one can prove the quasi-isometry invariance of the simplicial $L^p$-cohomology and then conclude that the de Rham $L^p$-cohomology is also invariant under some hypothesis. 

In the case of Orlicz cohomology, it is proved in \cite{C} the equivalence between both versions only in degree $1$. We present a proof in all degrees in the case of Lie groups equipped with left-invariant metrics. This proof has been obtained trying to improve the results of Pansu and Carrasco on the large scale geometry of Heintze groups, an special class of Lie groups that characterizes all connected homogeneous Riemannian manifolds with negative curvature. 

Finally, a relative version of the $L^p$-cohomology can simplify the computations in the case of Gromov-hyperbolic spaces (see \cite{S}), thus it can also be important to consider a relative version of the Orlicz cohomology and prove the de Rham's theorem in that context.

\subsection{Main definitions}

We say that a real function $\phi:\R\to [0,+\infty)$ is a \textit{Young function} if: 
\begin{itemize}
    \item it is even and convex; and
    \item $\phi(t)=0$ if, and only if, $t=0$.
\end{itemize}

Observe that every Young function $\phi$ satisfies
$$\lim_{t\to +\infty}\phi(t)=+\infty.$$

Let $(Z,\mu)$ be a measure space and $\phi$ a Young function. The \textit{Luxembourg norm associated with} $\phi$ of a measurable function $f:Z\to \overline{\R}=[-\infty,+\infty]$ is defined by 
$$\|f\|_{L^\phi}=\inf\left\{\gamma>0 : \int_Z\phi\left(\frac{f}{\gamma}\right)d\mu \leq 1\right\}\in [0,+\infty].$$
The \textit{Orlicz space of} $(Z,\mu)$ \textit{associated with }$\phi$ is the Banach space
$$L^\phi(Z,\mu)=\frac{\{f:Z\to \overline{\R}\text{ measurable}: \|f\|_{L^\phi}<+\infty\}}{\{f:Z\to \overline{\R}\text{ measurable}: \|f\|_{L^\phi}=0\}}.$$
It is not difficult to see that $\|f\|_{L^\phi}=0$ if, and only if, $f=0$ almost everywhere.  

If $\mu$ is the counting measure on $Z$, we denote $L^\phi(Z,\mu)=\ell^\phi(Z)$ and $\|\ \|_{L^\phi}=\|\ \|_{\ell^\phi}$. 
%If $Z$ is a topological space and $\mu$ is a Borel measure, then we can consider $L^{\phi,loc}(Z,\mu)$ as the space of measurable functions that are in $L^\phi(K,\mu)$ for all compact subset $K\subset Z$. 
Observe that if $\phi$ is the function $t\mapsto |t|^p$, then $L^\phi(Z,\mu)$ is the space $L^p(Z,\mu)$. We refer to \cite{RR} for a background about Orlicz spaces.

\begin{remark}\label{ObsEqivalenciaOrlicz}
If $K\geq 1$ is any constant, the identity map $\mathrm{Id}:L^{K\phi}(Z,\mu)\to L^{\phi}(Z,\mu)$ is clearly continuous and bijective, thus it is an isomorphism by the open mapping theorem. This implies that the norms $\|\text{ }\|_{L^{K\phi}}$ and $\|\text{ }\|_{L^\phi}$ are equivalent for all $K> 0$.
\end{remark}

We say that a simplicial complex $X$ equipped with a length distance has \textit{bounded geometry} if it has finite dimension and there exist a constant $C>0$ and a function $N:[0,+\infty)\to\N$ such that
\begin{enumerate}
    \item the diameter of every simplex is bounded by $C$;
    \item for every $r\geq 0$, the number of simplices contained in a ball of radius $r$ is bounded by $N(r)$.
\end{enumerate}

Denote by $X^k$ the set of $k$-simplices in $X$ and consider the cochain complex
\medskip
$$\ell^\phi(X^0)\stackrel{\delta}{\rightarrow} \ell^\phi(X^1) \stackrel{\delta}{\rightarrow} \ell^\phi(X^2) \stackrel{\delta}{\rightarrow}\ell^\phi(X^3) \stackrel{\delta}{\rightarrow}\cdots$$
\medskip
where $\delta$ is the usual coboundary operator ($\delta\theta(\sigma)=\theta(\partial\sigma)$). The \textit{Orlicz cohomology} of $X$ associated with the Young function $\phi$ (or, more simply, the $\ell^\phi$\textit{-cohomology} of $X$) is the family of topological vector spaces
$$\ell^\phi H^k(X)=\frac{\Ker \delta_k}{\I \delta_{k-1}}.$$
Since these spaces are not in general Banach spaces, it is something convenient to consider the \textit{reduced Orlicz cohomology} of $X$ associated with $\phi$ (or \textit{reduced} $\ell^\phi$\textit{-cohomology} of $X$) as the family of Banach spaces
$$\ell^\phi \overline{H}^k(X)=\frac{\Ker \delta_k}{\overline{\I \delta_{k-1}}}.$$

If $X$ is Gromov-hyperbolic and $\xi$ is a point on its boundary at infinity $\partial X$, then one can consider the \textit{relative Orlicz cohomology} of the pair $(X,\xi)$ associated with $\phi$ (or \textit{relative} $\ell^\phi$\textit{-cohomology} of $(X,\xi)$) as the family of topological vector spaces
$$\ell^\phi H^k(X,\xi)=\frac{\Ker \delta|_{\ell^\phi(X^k,\xi)}}{\I \delta|_{\ell^\phi(X^{k-1},\xi)}},$$
where $\ell^\phi(X^k,\xi)$ is the subspace of $\ell^{\phi}(X^k)$ that consists of all $k$-cochains that are zero on a neighborhood of $\xi$ in $\overline{X}=X\cup\partial X$. We say that a $k$-cochain on $X$ is zero or vanishes on a neighborhood of $\xi$ if there exists an open set $U$ in $\overline{X}$ which contains $\xi$, such that $\theta(\xi)=0$ for every $k$-simplex $\sigma$ contained in $U$. See for example \cite[Charpter 4]{BHK} for a description of the topology in $\overline{X}$.\\

%The relative $L^p$-cohomology is in some cases easier to compute than the (reduced) $L^p$-cohomology and it is sometimes a good way to find quasi-isometry invariants (see \cite{S}).\\

A map $F:X\to Y$ between two metric spaces (where the metric is denoted by $|\cdot-\cdot|$ in both cases) is a quasi-isometry if there exist two constants $\lambda\geq 1$ and $\epsilon\geq 0$ such that 
\begin{enumerate}
    \item For every $x_1,x_2\in X$,
    $$\lambda^{-1}|x_1-x_2|-\epsilon\leq |F(x_1)-F(x_2)|\leq \lambda|x_1-x_2|+\epsilon.$$
    \item For every $y\in Y$ there exists $x\in X$ such that $|F(x)-y|\leq\epsilon$.
\end{enumerate}

Every quasi-isometry $F:X\to Y$ admits a quasi-inverse. That is, quasi-isometry $\overline{F}:Y\to X$ such that $F\circ\overline{F}$ and $\overline{F}\circ F$ are at bounded uniform distance to the identity.

If $X$ and $Y$ are Gromov-hyperbolic, then the quasi-isometry $F$ induces a homeomorphism between their boundaries $\partial F:\partial X\to\partial Y$ (see for example \cite[Charpter 7]{GH}). We will use also the notation $F(\xi)=\partial F(\xi)$ if $\xi\in\partial X$.

We say that a metric space $X$ is \textit{uniformly contractible} if it is contractible and there is a function $\psi:[0,+\infty)\to [0,+\infty)$ such that every ball $B(x,r)=\{x'\in X : |x'-x|<r\}$ is contractible into the ball $B(x,\psi(r))$. 

\begin{theorem}\label{invarianza}
Let $X$ and $Y$ be two uniformly contractible simplicial complexes with bounded geometry and $\phi$ a Young function. If $X$ and $Y$ are quasi-isometric, then: 
\begin{enumerate}
    \item The cochain complexes $(\ell^\phi(X^*),\delta)$ and $(\ell^\phi(Y^*),\delta)$ are homotopically equivalent and as a consequence their cohomologies and reduced cohomologies are isomorphic.
    
    \item If $X$ is Gromov-hyperbolic and $\xi\in\partial X$, the cochain complexes $(\ell^\phi(X^*,\xi),\delta)$ and $(\ell^\phi(Y^*,F(\xi)),\delta)$ are homotopically equivalent and as a consequence their cohomologies are isomorphic.
\end{enumerate}
\end{theorem}

When we say \textit{homotopically equivalent} we mean a homotopy equivalence in a continuous sense, that is, all maps involved are continuous. Such a homotopy equivalence implies the isomorphism between the respective cohomology spaces in the sense of topological vector spaces.

The proof of the first part of Theorem \ref{invarianza} is done in \cite{C}. We will prove the second part in Section \ref{sec2} using the same method.\\ 

Now take a Riemannian manifold $M$ and denote by $\Omega^k(M)$ the space of all (smooth) differential $k$-forms on $M$. Consider
$$L^\phi\Omega^k(M)=\{\omega\in\Omega^k(M) : \|\omega\|_{L^\phi},\|d\omega\|_{L^\phi}<+\infty\},$$ 
equipped with the norm $|\omega|_{L^\phi}=\|\omega\|_{L^\phi}+\|d\omega\|_{L^\phi}$. Here $\|\omega\|_{L^\phi}$ is the Luxemburg norm of the function
$$x\mapsto|\omega|_x=\sup\{|\omega_x(v_1,\ldots,v_k)| : v_i\in T_xM \text{ for }i=1,\ldots,k, \text{ with } \|v_i\|_x=1\}$$
in the measure space $(M,dV)$, where $\|\ \|_x$ is the Riemannian norm on $T_xM$ and $dV$ is the Riemannian volume on $M$. We denote by $L^\phi C^k(M)$ the completion of $L^\phi\Omega^k(M)$ with respect to $|\text{ }\text{ }|_{L^\phi}$. Observe that the derivative of differential forms induces a continuous map $d=d_k:L^\phi C^k(M)\to L^\phi C^{k+1}(M)$.

We can consider the \textit{Orlicz-de Rham cohomology} of $M$ associated with $\phi$ (or $L^\phi$\textit{-cohomology} of $M$) as the family of topological vector spaces 
$$L^\phi H^k(M)=\frac{\Ker d_k}{\I d_{k-1}},$$
and the \textit{reduced Orlicz-de Rham cohomology} of $M$ (or \textit{reduced} $L^\phi$\textit{-cohomology} of $M$) as the family of Banach spaces
$$L^\phi \overline{H}^k(M)=\frac{\Ker d_k}{\overline{\I d_{k-1}}}.$$

\begin{remark}\label{ObsDef}
A measurable $k$-form on $M$ is a function $x\mapsto\omega_x$, where $\omega_x$ is an alternating $k$-linear form on the tangent space $T_xM$, such that the coefficients of $\omega$ for every parmetrization of $M$ are all measurable. We consider $L^\phi(M,\Lambda^k)$ the space of $L^\phi$-integrable measurable $k$-forms up to almost everywhere zero forms. It is a Banach space equipped with the Luxemburg norm $\|\ \|_{L^\phi}$. 

Since $L^\phi\Omega^k(M)\subset L^\phi(M,\Lambda^k)$ and the inclusion is continuous, one can prove using Hölder's inequality ($\|fg\|_{L^1}\leq 2\|f\|_{L^\phi}\|g\|_{L^{\phi^*}}$, where $\phi^*$ is the convex conjugate of $\phi$) that $L^\phi C^k(M)$ can be seen as a space of $k$-measurable forms in $L^\phi(M,\Lambda^k)$ which weak derivatives are defined and belong to $L^\phi(M,\Lambda^{k+1})$. 

We say that $\varpi$ is the weak derivative of $\omega\in L^\phi(M,\Lambda^k)$ if for every differential $(n-k)$-form with compact support $\alpha$ one has 
$$\int_M \varpi\wedge\alpha=(-1)^{k-1}\int_M \omega\wedge d\alpha.$$
\end{remark}

\medskip

If $M$ is Gromov-hyperbolic and $\xi\in\partial M$, we consider the subspace $L^\phi C^k(X,\xi)\subset L^\phi C^k(X)$ consisting of all $k$-forms that are zero in almost every point of a neighborhood of $\xi$. Then we can define the \textit{relative Orlicz-de Rham cohomology} of the pair $(M,\xi)$ associated with $\phi$ (or \textit{relative} $L^\phi$\textit{-cohomology} of $(M,\xi)$) as the family of topological vector spaces
$$L^\phi H^k(M,\xi)=\frac{\Ker d|_{L^\phi C^k(X,\xi)}}{\I d|_{L^\phi C^{k-1}(X,\xi)}}.$$

\subsection{Main result}

Consider a Lie group $G$ equipped with a left-invariant Riemannian metric. We denote by $dx$ the volume on $G$ and by $L_x$ and $R_x$ the left and right translation by $x\in G$ respectively.

Suppose that there exists a uniformly contractible simplicial complex $X$ with bounded geometry that is quasi-isometric to $G$. Then we can define the \textit{simplicial Orlicz cohomology}  and the \textit{reduced simplicial Orlicz cohomology} of $G$ as the families of spaces
$$\ell^\phi H^k(G)= \ell^\phi H^k(X)\text{ and }\ell^\phi \overline{H}^k(G)= \ell^\phi \overline{H}^k(X)$$
Observe that, because of Theorem \ref{invarianza}, it is well-defined up to isomorphisms.

If $G$ is Gromov-hyperbolic and $\xi\in\partial G$, we can consider the \textit{relative simplicial Orlicz cohomology} of the pair $(G,\xi)$ as the family of spaces
$$\ell^\phi H^k(G,\xi)= \ell^\phi H^k(X,\overline{\xi}),$$
where $\overline{\xi}$ is the image of $\xi$ by a quasi-isometry $F:G\to X$.

\begin{remark}\label{obssimplicial}
Let $M$ be a complete Riemannian manifold with bounded geometry. This means that it has positive injectivity radius and its sectional curvature is uniformly bounded from above and below. Assume that $\dim(M)=n$.

One can consider $X_M$ a triangulation of $M$ with bounded geometry such that every $n$-simplex is bi-Lipschitz diffeomorphic to the standard Euclidean simplex of the same dimension (see \cite{A}). For every vertex $v$ of $X_M$ we define $U(v)$ as the interior of the union of all simplices containing $v$. Observe that $X_M$ is the nerve of the covering $\mathcal{U}=\{U(v):v\in X_M^0\}$, and that every non empty intersection $U_1\cap\ldots\cap U_k$ of elements of $\mathcal{U}$ is bi-Lipschitz equivalent to the unit ball in $\R^n$ with uniform Lipschitz constant.

In general, we can consider $X_M$ as the nerve of any open covering satisfying the above properties and equip it with a length metric such that every simplex is isometric to the standard Euclidean simplex of the same dimension. It is clear that $X_M$ is quasi-isometric to $M$ in this case. Moreover, there is a family of quasi-isometries $F:X_M\to M$ verifying $F(U)\in U$ for all vertex $U\in\mathcal{U}$, we call them \textit{canonical quasi-ismometries}.

If $M$ is Gromov-hyperbolic and $\xi$ is a point in $\partial M$, observe that all canonical quasi-isometries are at bounded uniform distance from each other and as a consequence they induce the same map on the boundary. Denote by $\overline{\xi}\in X_M$ the point corresponding to $\xi$ by a canonical quasi-isometry. We say that $(X_M,\overline{\xi})$ is a \textit{simplicial pair corresponding} to $(M,\xi)$. As we saw with the first construction, if $M$ is uniformly contractible we can suppose that $X_M$ is also uniformly contractible.

Since a Lie group $G$ equipped with a left-invariant metric is always complete and has bounded geometry, then one can consider the simplicial complex $X_G$. If $G$ is in addition contractible, then it is uniformly contractible and its (relative/reduced) simplicial Orlicz cohomology is well-defined.
\end{remark}

Our main result is the following:

\begin{theorem}\label{main}
Let $G$ be a Lie group equipped with a left-invariant metric and $X_G$ a simplicial complex as in Remark  \ref{obssimplicial}, then
\begin{enumerate}
    \item The (reduced) $L^\phi$-cohomology of $G$ and the (reduced) $\ell^\phi$-cohomology of $X_G$ are isomorphic.
    
    \item If $G$ is Gromov-hyperbolic and $\xi\in\partial G$, the relative $L^\phi$-cohomology of the pair $(G,\xi)$ and the relative $\ell^\phi$-cohomology of the pair $(X_G,\overline{\xi})$ are isomorphic.
\end{enumerate}
\end{theorem}

As a consequence of the proof of Theorem \ref{main} we will obtain:

\begin{theorem}\label{Kop}
If $G$ is a Lie group equipped with a left-invariant metric, then the cochain complexes $(L^\phi\Omega^k(G),d)$ and $(L^\phi C^k(G),d)$ are homotopically equivalent. The same result is true for the relative complexes in the Gromov-hyperbolic case.
\end{theorem}

A more general version of Theorem \ref{Kop} is proved in \cite{KP}.\\

Theorem \ref{main} implies that if $G$ is contractible and $\phi$ is a Young function, then $\ell^\phi H^k(G)$ is isomorphic to $L^\phi H^k(G)$ and $\ell^\phi \overline{H}^k(G)$ is isomorphic to $L^\phi \overline{H}^k(G)$. If $G$ is in addition Gromov-hyperbolic and $\xi$ is a point in $\partial G$, then $\ell^\phi H^k(G,\xi)$ is isomorphic to $L^\phi H^k(G,\xi)$.

Combining this with Theorem \ref{invarianza} we get:

\begin{corollary}\label{main3}
If $F:G_1\to G_2$ is a quasi-isometry between two contractible Lie groups equipped with left-invariant metrics and $\phi$ is a Young function, then for every $k\in\N$
\begin{itemize}
    \item the topological vector spaces $L^\phi H^k(G_1)$ and $L^\phi H^k(G_2)$ are isomorphic; and
    \item the Banach spaces $L^\phi \overline{H}^k(G_1)$ and $L^\phi \overline{H}^k(G_2)$ are isomorphic.
\end{itemize}
Furthermore, if $G_1$ and $G_2$ are Gromov-hyperbolic and $\xi$ is a point in $\partial G_1$, then the spaces $L^\phi H^k(G_1,\xi)$ and $L^\phi H^k(G_2,F(\xi))$ are isomorphic for every $k$.
\end{corollary}

\section{Invariance of the simplicial relative case}\label{sec2}

Let $X$ be a simplicial complex with bounded geometry and fix a Young function $\phi$. Observe that every element $\theta\in \ell^\phi(X^k)$ has a natural linear extension $\theta:C_k(X)\to\R$, where
$$C_k(X)=\left\{\sum_{i=1}^m t_i\sigma_i :  t_1,\ldots,t_m\in\R, \sigma_1,\ldots,\sigma_m\in X^k\right\}.$$ 
The \textit{support} of a chain $c=\sum_{i=1}^m t_i\sigma_i$ in $C_k(X)$ (with $t_i\neq 0$ for all $i=1,\ldots,m$) is $|c|=\{\sigma_1,\ldots,\sigma_m\}$. We also define the \textit{uniform norm} and the \textit{length} of $c$ by
$$\|c\|_\infty=\max\{|t_1|,\ldots,|t_m|\}\text{, and } \ell(c)=m.$$

\begin{proposition}\label{ContinuidadBordeOrlicz}
The usual coboundary operator $\delta=\delta_k:\ell^\phi (X^k)\to\ell^\phi (X^{k+1})$ is continuous.
\end{proposition}

\begin{proof}
Let $\theta$ be a cochain in $\ell^\phi (X^k)$, then
\begin{align*}
    \|\delta\theta\|_{\ell^\phi} %&=\inf\left\{\gamma>0 : \sum_{\sigma\in X_{k+1}} \phi\left(\frac{\delta\theta(\sigma)}{\gamma}\right)\leq 1\right\}\\
                            &=\inf\left\{\gamma>0 : \sum_{\sigma\in X_{k+1}} \phi\left(\frac{\theta(\partial\sigma)}{\gamma}\right)\leq 1\right\}.
\end{align*}
The bounded geometry implies that there is a constant $N(1)$ such that every $k$-simplex $\tau$ in $X$ is on the boundary of at most $N(1)$ $(k+1)$-simplices. Then 
$$\sum_{\sigma\in X_{k+1}} \phi\left(\frac{\theta(\partial\sigma)}{\gamma}\right)\leq \sum_{\tau\in X_k} N(1) \phi\left(\frac{\theta(\tau)}{\gamma}\right),$$
which implies
$$\|\delta\theta\|_{\ell^\phi}\leq \inf\left\{\gamma>0 : \sum_{\tau\in X_k} N(1) \phi\left(\frac{\theta(\tau)}{\gamma}\right)\leq 1\right\}= \|\theta\|_{\ell^{N(1)\phi}}.$$
The proof ends using the equivalence between $\|\text{ }\|_{\ell^{N(1)\phi}}$ and $\|\text{ }\|_{\ell^\phi}$ (Remark \ref{ObsEqivalenciaOrlicz}).
\end{proof}

To prove Theorem \ref{invarianza} we need the following lemmas.

\begin{lemma}[\cite{BP}]\label{lema1} 
Let $X$ and $Y$ be two uniformly contractible simplicial complexes with bounded geometry. Then any quasi-isometry $F:X\to Y$ induces a family of maps $c_F:C_k(X)\to C_k(Y)$ which verify:
\begin{enumerate}
\item[(i)] $\partial c_F(\sigma)=c_F(\partial\sigma)$ for every $\sigma\in X^k$.
\item[(ii)] For every $k\in\N$ there exist constants $N_k$ and $L_k$ (depending only on $k$ and the geometric data of $X,Y$ and $F$) such that 
$\|c_F(\sigma)\|_\infty\leq N_k$ and $\ell(c_F(\sigma))\leq L_k$ for every $\sigma\in X^k$
\end{enumerate}
Furthermore, the Hausdorff distance between $c_F(\sigma)$ and $F(\sigma)$ is uniformly bounded. 
\end{lemma}

%\begin{proof}
%We consider for $X$ and $Y$ the same constant $C\geq 0$ and function $N:[0,+\infty)\to\N$ corresponding to their bounded geometry. We assume also that both spaces are uniformly contractible for the same function $\phi$.

%\end{proof}

\begin{lemma}[\cite{BP}]\label{lema2}
Consider $F,G:X\to Y$ two quasi-isometries between uniformly contractible simplicial complexes with bounded geometry. If $F$ and $G$ are at bounded uniform distance, then there exists an homotopy $h:C_k(X)\to C_{k+1}(Y)$ between $c_F$ and $c_G$. This means that
\begin{enumerate}
\item[(i)] $\partial h(v)=c_F(v)-c_G(v)$ if $v\in X^0$, and
\item[(ii)] $\partial h(\sigma)+h(\partial\sigma)=c_F(\sigma)-c_G(\sigma)$ if $\sigma\in X^k$, $k\geq 1$.
\end{enumerate}
Moreover, $\|h(\sigma)\|_\infty$ and $\ell(h(\sigma))$ are uniformly bounded by constants $N'_k$ and $L'_k$ that depend only on $k$ and the geometric data of $X,Y,F$ and $G$.
\end{lemma}

\medskip

\begin{proof}[Proof of Theorem \ref{invarianza} (part 2)]

We define the pull-back of a cochain $\theta\in\ell^\phi(Y^k,F(\xi))$ as the composition $F^*\theta=\theta\circ c_F$. The map $c_F$ given by Lemma \ref{lema1} is not unique, then $F^*$ depends on the choice of it. 

Let us prove that $F^*:\ell^\phi(Y^k,F(\xi))\to \ell^\phi(X^k,\xi)$ is well-defined and continuous:
\begin{align*}
\|F^*\theta\|_{\ell^\phi}
    & =\inf\left\{\gamma>0 : \sum_{\sigma\in X^k}\phi\left(\frac{\theta(c_F(\sigma))}{\gamma}\right)\leq 1\right\}\\
    & \leq\inf\left\{\gamma>0 : \sum_{\sigma\in X^k}\phi\left(\frac{N_k}{\gamma}\sum_{\tau\in |c_F(\theta)|}|\theta(\tau)|\right)\leq 1\right\}\\
    & \leq\inf\left\{\gamma>0 : \sum_{\sigma\in X^k}\sum_{\tau\in |c_F(\theta)|}\frac{1}{\ell(c_F(\sigma))}\phi\left(\frac{N_kL_k}{\gamma}|\theta(\tau)|\right)\leq 1\right\},
\end{align*}
where $N_k$ and $L_k$ are the constants given by Lemma \ref{lema1}.

Since $F$ is a quasi-isometry and the Hausdorff distance between $c_F(v)$ and $F(v)$ is uniformly bounded for all $v\in X^0$, we can find a constant $C_k$ such that if $dist(\sigma_1,\sigma_2)>C_k$, then $c_F(\sigma_1)\cap c_F(\sigma_2)=\emptyset$. Using the bounded geometry of $X$ we have that every $\tau\in Y^k$ satisfies $\tau\in|c_F(\sigma)|$ for at most $D=N(C+C_k)$ simplices $\sigma\in X^k$. This implies
\begin{align*}\label{contiuidad}
\|F^*\theta\|_{\ell^\phi}
    & \leq\inf\left\{\gamma>0 : \sum_{\sigma\in Y^k}D\phi\left(\frac{N_kL_k}{\gamma}|\theta(\tau)|\right)\leq 1\right\}=N_kL_k\|\theta\|_{\ell^{C\phi}}\preceq\|\theta\|_{\ell^\phi}.
\end{align*}
Hence $F^*\theta\in\ell^\phi (X^k)$. We write $f\preceq g$ for a pair of non-negative functions $f$ and $g$ if there exists a constant $K$ such that $f\leq Kg$. 

Now we prove that for every $\theta$ in $\ell^\phi(Y^k,{F(\xi)})$, the cochain $F^*\theta$ is zero on some neighborhood of $\xi$. Assume that $\theta$ is zero on $V\subset \overline{Y}$, $F(\xi)\in V$. If $\sigma\in X^k$ and $v\in X^0$ is a vertex of $\sigma$,
\begin{equation}\label{distHausdorff}
d_H(c_F(\sigma),F(v))\leq d_H(c_F(\sigma),c_F(v))+d_H(c_F(v),F(v)),    
\end{equation}
where $d_H$ denotes the Hausdorff distance. By the properties of $c_F$ the distance (\ref{distHausdorff}) is uniformly bounded by a constant $\tilde{C}_k$. We define $\tilde{V}=\{y\in Y: dist(y,V^c\cap Y)>\tilde{C}_k\}$. Since $F$ is a quasi-isometry, there exists $U\subset\overline{X}$ a neighbourhood of $\xi$ such that $F(U\cap X)\subset \tilde{V}$. For every $k$-simplex $\sigma\subset U$, we have $c_F(\sigma)\subset V$ and then $F^*\theta(\sigma)=0$. We conclude that $F^*\theta$ vanishes on $U$.

Since $c_F$ commutes with the boundary, we have $\delta F^*=F^*\delta$, which implies that $F^*$ defines a continuous map in cohomology, denoted by $F^{\#}:\ell^\phi H^k(Y,F(\xi))\to \ell^\phi H^k(X,\xi)$. We have to prove that $F^{\#}$ is an isomorphism.\\

\underline{Claim:} If $F,G:X\to Y$ are two quasi-isometries at bounded uniform distance, then $F^{\#}=G^{\#}$.\\

We have to construct a family of continuous linear maps $H_k:\ell^\phi(Y^k,F(\xi))\to \ell^\phi(X^{k-1},\xi)$ such that:
\begin{enumerate}
    \item[(i)] $F^*\theta-G^*\theta=H_1\delta\theta$ for every $\theta\in\ell^\phi(Y^0,F(\xi))$, and
    \item $F^*\theta-G^*\theta=H_{k+1}\delta\theta+\delta H_k\theta$ for every $\theta\in\ell^\phi(Y^k,F(\xi))$, $k\geq 1$.
\end{enumerate}

We define $H_k\theta:X^k\to\R$, $H_k\theta(\sigma)=\theta(h(\sigma))$, where $h$ is the map given by Lemma \ref{lema2}. Using the same argument as for $F^*$, one can show that $H_k$ is well-defined and continuous from $\ell^\phi(Y^k,F(\xi))$ to $\ell^\phi(X^k)$. To see that $H_k\theta$ vanishes on some neighborhood of $\xi$ observe that $h(\sigma)$ have uniformly bounded length, which implies that $d_H(c_F(\sigma),h(\sigma))$ is uniformly bounded.

Using the definition of $H_k$ one can easily verify $(i)$ and $(ii)$, which proves the claim.\\

As a consequence of the claim we have that $F^{\#}$ does not depend on the choice of $c_F$. Moreover, if $T:Y\to Z$ is another quasi-isometry, a possible choice of the function $c_{T\circ F}$ is the composition $c_T\circ c_F$. In this case $(T\circ F)^*=F^*\circ T^*$ and as a consequence $(T\circ F)^{\#}=F^{\#}\circ T^{\#}$.

Finally, if $\overline{F}:Y\to X$ is a quasi-inverse of $F$, then by the claim $(F\circ \overline{F})^{\#}$ and $(\overline{F}\circ F)^{\#}$ are the identity in cohomology. Since $(F\circ \overline{F})^{\#}= \overline{F}^{\#}\circ F^{\#}$ and $(\overline{F}\circ F)^{\#}=F^{\#}\circ\overline{F}^{\#}$, the statement follows.
\end{proof}

\section{Integration and convolution of forms}\label{Cap4Sec2S1}

Suppose that $M$ is a smooth manifold of dimension $n$ and $(Z,\mu)$ is a measure space. We say that  $\Phi=\{\Phi_{(x,z)}:x\in M, z\in Z\}$ is a \textit{family of measurable $k$-forms on} $M$ if for every $(x,z)\in M\times Z$, $\Phi_{(x,z)}$ is an alternating $k$-form on the tangent space $T_x M$ and all coefficients of $\Phi$ with respect to every parametrization (depending on $x\in M$ and $z\in Z$) are measurable. It is a \textit{smooth family of $k$-forms} if its coefficients are smooth.

We say that $\Phi$ is \textit{integrable on} $Z$ if for every $x\in M$, the function
$$z\mapsto |\Phi|_{(x,z)}=\sup\{|\Phi_{(x,z)}(v_1,\ldots,v_k)| : v_i\in T_xM\text{ for }i=1,\ldots,k,\text{ with }\|v_i\|_x=1\}$$
belongs to $L^1(Z,\mu)$. In this case we can consider the $k$-form
\begin{equation}\label{formaIntegral}
\omega_x(v_1,\ldots,v_k)=\left(\int_{Z}\Phi_{(x,z)}d\mu(z)\right)(v_1,\ldots,v_k)=\int_{Z}\Phi_{(x,z)}(v_1,\ldots,v_k)d\mu(z).    
\end{equation}
Observe that for all $x\in M$,
$$|\omega|_x\leq \int_Z |\Phi|_{(x,z)}d\mu(z)=\|\Phi_{(x,\cdot)}\|_{L^1}.$$

\begin{lemma}\label{lemaDerivadaDebil}
Let $\{\Phi_{(x,z)}:x\in M,z\in Z\}$ be a measurable family of $k$-forms such that:
\begin{itemize}
\item It is integrable on $Z$, then we can define $\omega$ as in (\ref{formaIntegral}).
\item For every fixed $z\in Z$ the $k$-form $x\mapsto \Phi_{(x,z)}$ is locally integrable and has weak derivative $d\Phi_{(x,z)}$.
\item The function $z\mapsto |d\Phi|_{(x,z)}$ belongs to $L^1(Z,\mu)$ for every $x\in M$.
\end{itemize}
Then $\omega$ is locally integrable and has weak derivative 
\begin{equation}\label{derivadaintegral}
d\omega_x=\int_Z d\Phi_{(x,z)}d\mu(z).
\end{equation}
\end{lemma} 

The previous lemma follows directly from definition of weak derivative. 

\medskip

To prove that a measurable $k$-form $\omega$ on $M$ is smooth it is enough to verify that for every set of $k$ vector fields $\{X_1,\ldots,X_k\}$ the function
$$f(x)=\omega_x(X_1(x),\ldots,X_k(x))$$
is smooth on $M$. A sufficient condition for $f$ to be smooth is that for every set of vector fields $\{Y_1,\ldots,Y_m\}$ there exists
$$L_{Y_m}\cdots L_{Y_1}f (x)$$
for all $x\in M$.
The \textit{Lie derivative} with respect to the field $Y$ is defined by
$$L_Y f (x) = \left.\frac{\partial}{\partial t}\right|_{t=0}f(\varphi_t(x)),$$
where $\varphi_t$ is the flow associated with $Y$.

From the above observation and the classical Leibniz Integral Rule one can conclude the following lemma:

\begin{lemma}\label{Leibniz}
Let $M$ and $N$ be two Riemannian manifolds and $\{\Phi_{(x,z)}:x\in M,z\in Z\}$ a smooth family $k$-forms on $M$. If $\Phi_{(x,y)}$ has compact support for every $y\in N$, then the $k$-form on $M$ defined by
$$\omega_x=\int_N\Phi_{(x,y)}dV_N(y)$$
%$$\omega_x(v_1,\ldots,v_k)=\left(\int_N \Phi_{(x,y)} dV_N(y)\right)(v_1,\ldots,v_k)=\int_N \Phi_{(x,y)}(v_1,\ldots,v_k)dV_N(y)$$
belongs to $\Omega^k(M)$ and its derivative is $d\omega=\int_N d\Phi(\cdot,y) dV_N(y).$
\end{lemma}

\medskip
Now consider a Lie group $G$. By a  \textit{kernel} on $G$ we mean a smooth function $\kappa:G\to[0,1]$ such that:
\begin{itemize}
    \item $\supp(\kappa)$ is a compact neighborhood of $e\in G$, and
    \item $\int_G \kappa(x)dx=1$.
\end{itemize}

If $\omega$ is a locally integrable $k$-form on $G$ we consider its \textit{convolution} with $\kappa$ as the $k$-form
$$(\omega *\kappa)_x=\int_G (R_z^*\omega)_x\kappa(z)dz.$$

\medskip

\begin{lemma}\label{lemaNormaConvolucion}
There exists a constant $C>0$ such that for every locally integrable $k$-form $\omega$ on $G$ and $x\in G$ we have
$$|\omega*\kappa|_x\leq C |\omega|*\kappa(x),$$
where $|\omega|*\kappa$ is the convolution of the function $x\mapsto|\omega|_x$ with the kernel $\kappa$.
\end{lemma}

\begin{proof}
Let $v_1,\ldots,v_k$ be vectors in $T_xG$, then
\begin{align*}
|(\omega*\kappa)_x(v_1,\ldots,v_k)| 
    &= \left|\int_G(R_z^*\omega)_x(v_1,\ldots,v_k)\kappa(z)dz\right|\\
    &\leq \int_G|(R_z^*\omega)_x(v_1,\ldots,v_k)|\kappa(z)dz\\
    &= \int_G|\omega_{x\cdot z}(d_xR_z(v_1),\ldots,d_xR_z(v_k))|\kappa(z)dz.
\end{align*}
Since $R_z\circ L_x=L_x\circ R_z$, we have $|d_e(R_z\circ L_x)|=|d_e(L_x\circ R_z)|$ (here $|\ \ |$ is the usual operator norm) and therefore $|d_x R_z\circ d_eL_x|=|d_zL_x\circ d_eR_z|$. Using that $L_x$ is an isometry we obtain $|d_xR_z|=|d_eR_z|$ for every $x\in G$. The function $z\mapsto|d_eR_z|$ is continuous, then it has a maximum $M$ in $\supp(\kappa)$. If $\|v_1\|=\ldots=\|v_k\|=1$,
$$|\omega_{x\cdot z}(d_xR_z(v_1),\ldots,d_xR_z(v_k))|\leq M^k|\omega|_{x\cdot z},$$
which implies $|\omega*\kappa|_x\leq C |\omega|*\kappa(x)$ with $C=M^k$.
\end{proof}

A consequence of Lemma \ref{lemaNormaConvolucion} is that the convolution of a locally integrable form is also locally integrable.

\begin{proposition}\label{PropDerConv}
Let $\omega$ be a locally integrable $k$-form on $G$, then:
\begin{enumerate}
    \item[(i)] If $\omega$ has weak derivative $d\omega$, then the convolution $\omega*\kappa$ has weak derivative and
    $$d(\omega*\kappa)=d\omega*\kappa.$$
    \item[(ii)] The convolution $\omega*\kappa$ is a differential form.
\end{enumerate}
\end{proposition}

\begin{proof}
\begin{enumerate}
\item[(i)] 
For every $z$ we have $d(R_z^*\omega)=R_z^*d\omega$ in a weak sense. To see this take $\beta\in\Omega^{n-k-1}(G)$ with compact support, then
\begin{align*}
\int_G(R_z^*d\omega)\wedge\beta &= \int_G R_z^*(d\omega\wedge R^*_{z^{-1}}\beta)\\
&= \int_G d\omega\wedge R^*_{z^{-1}}\beta\\
&= (-1)^{k+1}\int_G\omega\wedge d R^*_{z^{-1}}\beta\\
&=  (-1)^{k+1}\int_G(R^*_z\omega)\wedge d\beta.
\end{align*}

Therefore the weak derivative with respect to $x\in G$ of the $k$-form $\Phi_{(x,z)}=(R_z^*d\omega)_x\kappa(z)$ is 
$$d\Phi(x,z)=(R^*_z d\omega)_x\kappa(z).$$
Since $z\mapsto d\Phi(x,z)$ has compact support for all $x\in G$, by Lemma \ref{lemaDerivadaDebil} we conclude 
$$(d\omega*\kappa)=\int_G(R^*_zd\omega)\kappa(z)dz$$
is the weak derivative of the convolution $\omega*\kappa$.

\item[(ii)]
Suppose first that $\omega=f$ is a $0$-form, which is equivalent to say that it is a locally integrable function on $G$. Consider $Y$ a vector field on $G$ with flow $\varphi_t$. First observe that
$$f*\kappa(x)=\int_G f(x\cdot z)\kappa(z)dz=\int_G f(y)\kappa(x^{-1}\cdot y)dy.$$
Then
\begin{align*}
L_{Y} (f*\kappa)(x)
&= \left.\frac{\partial}{\partial t}\right|_{t=0}(f*\kappa(\varphi_t(x)))\\
&=\left.\frac{\partial}{\partial t}\right|_{t=0} \int_G f(y)\kappa(\varphi_t(x)^{-1}\cdot y)dy.
\end{align*}
Since $\varphi$ is smooth and $\kappa$ is smooth with compact support, the classical Leibniz integral Rule implies that this derivative exists and
\begin{align*}
L_{X} (f*\kappa)(x)
&= \int_G f(y)\left.\frac{\partial}{\partial t}\right|_{t=0}\kappa(\varphi_t(x)^{-1}\cdot y)dy. 
\end{align*}
Using this argument we can prove by induction that $L_{Y_m}\ldots L_{Y_m}(f*\kappa)(x)$ exists for all $x\in G$ and every family of vector fields $Y_1,\ldots,Y_m$, which implies that $f*\kappa$ is smooth.

Now consider $\{e_1,\ldots,e_n\}$ a basis of $T_e G$ and $X_1,\ldots,X_n$ the right-invariant fields verifying $X_i(e)=e_i$. Let $\varphi_t^i$ be the flow associated with $X_i$ for every $i=1,\ldots,n$. If $\omega$ is a $k$-form with $k\geq 1$ we set
$$f_{i_1,\ldots,i_k}(x)=(\omega*\kappa)_x(X_{i_1}(x),\dots,X_{i_k}(x)).$$
To prove that $\omega*\kappa$ is smooth it is enough to prove that all these functions are smooth. Observe that if 
$$g_{i_1,\ldots,i_k}(x)=\omega(x)(X_{i_1}(x),\dots,X_{i_k}(x)),$$
then
$f_{i_1,\ldots,i_k}=g_{i_1,\ldots,i_k}*\kappa$. This reduces the general case to the case $k=0$ and finishes the proof.
\end{enumerate}
\end{proof}

The last lemma of the section relates the $L^\phi$-norm with the $L^1$-norm in the case of finite measure.

\begin{lemma}\label{OrliczL1}
If $\mu$ is finite, then $L^\phi(Z,\mu)\subset L^1(Z,\mu)$ and the inclusion is continuous, with norm bounded depending only on $\mu(Z)$ and $\phi$.
\end{lemma}

\begin{proof} Let $f\in L^\phi(Z,\mu)$, then
\begin{align*}
\|f\|_{L^\phi}
    &=\inf\left\{\gamma>0 : \int_Z\phi\left(\frac{f}{\gamma}\right)d\mu\leq 1\right\}\\
    &\geq \inf\left\{\gamma>0 : \mu(Z)\phi\left(\frac{1}{\mu(Z)}\int_Z\frac{f}{\gamma}d\mu\right)\leq 1\right\}
\end{align*}
From this we obtain $\|f\|_{L^1}\leq \mu(Z)\phi^{-1}(1/\mu(Z))\|f\|_{L^\phi}$.
\end{proof}

\section{Proof of main theorem}

Let $\mathcal{U}$ be an uniformly locally finite open covering on the Lie group $G$ such that every non-empty intersection $U_1\cap\cdots\cap U_k$ is uniformly bi-Lipschitz to the unit Euclidean ball. By \textit{uniformly locally finite} we mean that there exists a uniform constant $C$ such that every point in $G$ belongs to at most $C$ elements of $\mathcal{U}$. Take $X_G$ as in the Remark \ref{obssimplicial}, that is, $$X^\ell=\mathcal{U}_\ell=\{U_0\cap\cdots\cap U_k\neq\emptyset : U_0,\ldots,U_\ell\in\mathcal{U}\},$$
and every simplex is isometric to the standard one of the same dimension.

We consider, for a fixed Young function $\phi$, the following cochain complexes:

\begin{itemize}
    \item $\mathcal{L}^\phi \Omega^k(G,\mathcal{U})$ is the space of all differential forms $\omega\in\Omega^k(G)$ such that $\omega|_U$ and $d\omega|_U$ are in $L^\phi\Omega^k(U)$ for every $U\in\mathcal{U}$, and the functions $U\mapsto \|\omega|_U\|_{L^\phi}$ and   $U\mapsto\|d\omega|_U\|_{L^\phi}$ are in $\ell^\phi(\mathcal{U})$. The norm of $\omega\in\mathcal{L}^\phi \Omega^k(G,\mathcal{U})$ is defined by
    $$|\omega|_{\mathcal{L}^\phi}=\|\theta\|_{\ell^\phi}+\|\theta'\|_{\ell^\phi},$$
    where $\theta(U)=\|\omega|_U\|_{L^\phi}$ and $\theta'(U)=\|d\omega|_U\|_{L^\phi}$. Naturally, the map defining the cochain complex is the usual derivative.
    
    \item $\mathcal{I}^\phi\Omega^k(G,\mathcal{U})=L^\phi\Omega^k(G)\cap \mathcal{L}^\phi\Omega^k(G,\mathcal{U})$ with the norm  $|\ \ |_{\mathcal{I}^\phi}=|\ \ |_{L^\phi}+|\ \ |_{\mathcal{L}^\phi}.$
    
    \item If $G$ is Gromov-hyperbolic and $\xi$ is a point in $\partial G$, we consider $\mathcal{L}^\phi \Omega^k(G,\mathcal{U},\xi)$ the subcomplex consisting of all forms $\omega\in\mathcal{L}^\phi \Omega^k(G,\mathcal{U})$ that vanish on a neighborhood of $\xi$. In this case we also define $\mathcal{I}^\phi\Omega^k(G,\mathcal{U},\xi)=L^\phi\Omega^k(G,\xi)\cap \mathcal{L}^\phi\Omega^k(G,\mathcal{U},\xi)$.
\end{itemize}

If $\phi(t)=|t|^p$ for every $t\in\R$, then $L^\phi\Omega^k(G)=\mathcal{L}^\phi\Omega^k(G,\mathcal{U})=\mathcal{I}^\phi\Omega^k(G,\mathcal{U})$ and the norms on these spaces are equivalent. However, it is not true for general Young functions, as one can see in the following example:

\begin{example}\footnote{This example was given to me by Marc Bourdon.}
We take the Young function $\phi:\R\to[0,+\infty)$,
$$\phi(t)=\phi_{p,\kappa}(t)=\frac{|t|^p}{\log(e+|t|^{-1})^\kappa},$$
with $p>1$ and $\kappa>0$. This is a \textit{doubling Young function}, which means that there is a constant $D$ such that $\phi(2t)\leq D\phi(t)$ for every $t\in\R$. This condition implies some nice properties of the corresponding Orlicz space.

We want to construct a $1$-form $\omega$ in $\mathcal{L}^\phi\Omega^1(\R,\mathcal{U})$ and out of $L^\phi\Omega^1(\R)$, where 
$\mathcal{U}=\{U_n=(n-\epsilon,n+1+\epsilon) : n\in\Z\}$ with $\epsilon>0$ much smaller than 1.

Let $\{a_n\}_{n\in\Z}$ be a sequence of positive numbers such that:
\begin{itemize}
    \item $\sum a_n^p=+\infty$, and
    \item $\sum\phi(a_n)<+\infty$.
\end{itemize}
Take for every $n\in \Z$ an interval $A_n$ in $\R$ such that $A_n\subset (n+2\epsilon,n+1-2\epsilon)$ and $long(A_n)=a_n^p$ (we can suppose that $a_n$ is small enough for every $n\in\Z$). Consider the function $f:\R\to\R$ defined by 
$$f=\sum_{n\in\Z} \mathbbm{1}_{A_n}.$$ 

On the one hand, if $\gamma>0$,
$$\int_\R\phi\left(\frac{f(t)}{\gamma}\right)dt=\sum_{n\in\Z}\int_n^{n+1}\phi\left(\frac{\mathbbm{1}_{A_n}(t)}{\gamma}\right)dt=\sum_{n\in\Z}a_n^p\phi\left(\frac{1}{\gamma}\right)=+\infty.$$
But on the other hand
$$\int_{U_n}\phi\left(\frac{f(t)}{\gamma}\right)dt=a_n^p\phi\left(\frac{1}{\gamma}\right)=\frac{a_n^p}{\gamma^p\log(e+\gamma)^\kappa}\leq \left(\frac{a_n}{\gamma}\right)^p,$$
which implies $\|f|_{U_n}\|_{L^\phi}\leq a_n$ and then
\begin{equation}\label{normafinita}
\sum_{n\in\Z}\phi(\|f|_{U_n}\|_{L^\phi})\leq \sum_{n\in\Z}\phi(a_n)<+\infty.    
\end{equation}
It is not difficult to see, using the doubling condition, that \eqref{normafinita} implies that $\{\|f|_{U_n}\|\}_{n\in\Z}$ belongs to $\ell^\phi(\Z)$.

We can find a smooth function $g$ close enough from $f$ such that $g-f\in L^\phi(\R)$ and $\{\|(g-f)|_{U_n}\|_{L^\phi}\}_{n\in\Z}\in\ell^\phi(\Z)$ and consider the $1$-form $\omega=g\ dt$. Since $|\omega|_t=|g(t)|$ and $d\omega=0$ we can see that $\omega\in\mathcal{L}^\phi\Omega^1(\R,\mathcal{U})$ and $\omega\notin L^\phi\Omega^1(\R)$.

In this case the other inclusion is true. One can prove that $L^\phi\Omega^k(\R)\subset\mathcal{L}^\phi\Omega^k(\R,\mathcal{U})$ for $k=0,1$ using the inequality $\phi(s)\phi(t)\leq 2^\kappa\phi(st)$. In fact, this inclusion can be proved for every Riemannian manifold with bounded geometry and every doubling Young function satisfying an inequality $\phi(t)\phi(s)\leq C\phi(st)$ for all $s,t\in\R$ and some constant $C$.
\end{example}

We will prove Theorem \ref{main} in three steps. 

\begin{proposition}[First step]\label{PropEq1}
The cochain complexes $(\ell^\phi C^*(X_G),\delta)$ and $(\mathcal{L}^\phi\Omega^*(G,\mathcal{U}),d)$ are homotopically equivalent. So are the relative cochain complexes $(\ell^\phi C^*(X_G,\bar{\xi}),\delta)$ and $(\mathcal{L}^\phi\Omega^*(G,\mathcal{U},\xi),d)$. 
\end{proposition}

To prove the proposition we need some lemmas. The first one is a $L^\phi$-version of Lemma 8 in \cite{Pa95}.

\begin{lemma}\label{alciclicoOrlicz}
Let $B$ be the unit ball in the Euclidean space $\R^n$.
The cochain complex $(L^\phi \Omega^*(B),d)$ retracts to the complex $(\R\to 0\to 0\to \ldots)$.
\end{lemma}

\begin{proof}
Fix $x\in B$. Suppose that $\chi:\Omega^k(B)\to\Omega^{k-1}(B)$ is defined for all $k\geq 1$ so that for every $(k-1)$-simplex $\tau\subset B$, we have
$$\int_\tau\chi(\omega)=\int_{C_\tau}\omega$$
for every differential $k$-form $\omega$. The cone $C_\tau$ is defined as follows: If $\tau=(x_0,\ldots,x_{k-1})$, then $C_\tau=(x,x_0,\ldots,x_{k-1})$. The function $\chi$ will depend on $x$, we write $\chi_x=\chi$ if necessary.\\

\underline{Claim}:
\begin{equation}\label{dos} 
\chi d+d\chi=\mathrm{Id}.
\end{equation}

\medskip

Take $\sigma$ a $k$-simplex in $B$ and $\omega\in\Omega^k(B)$, then
$$
\int_\sigma \chi(d\omega) = \int_{C_\sigma} d\omega = \int_{\partial C_\sigma}\omega,$$
where the last equality comes from Stokes' theorem. If $\partial\sigma=\tau_0+\ldots+\tau_k$, we have
\begin{align*}
\int_\sigma \chi(d\omega) 
    &=\int_\sigma \omega-\sum_{i=0}^k \int_{C_{\tau_i}}\omega =\int_\sigma \omega- \int_{\partial\sigma}\chi(\omega)=\int_\sigma \omega- \int_{\sigma}d\chi(\omega).
\end{align*}
Since the equality holds for every $k$-simplex we conclude (\ref{dos}) (see for example \cite[Chapter IV]{W}).\\

For $x\in B$ we consider $\varphi=\varphi_x:[0,1]\times B\to B$, $\varphi_x(t,y)=ty+(1-t)x$ and $\eta_t:B\to [0,1]\times B$, $\eta_t(y)=(t,y)$. We look for an explicit expression for $\chi(\omega)$:
\begin{align*}
\int_\sigma \chi(\omega) = \int_{C_\sigma}\omega = \int_{\varphi([0,1]\times\sigma)}\omega = \int_{[0,1]\times\sigma}\varphi^*\omega = \int_\sigma\int_0^1 \eta_s^*(\iota_{\frac{\partial}{\partial t}}\varphi^*\omega)ds.
\end{align*}
Where $\frac{\partial}{\partial t}$ is the vector field on $[0,1]\times B$ defined by $\frac{\partial}{\partial t}(s,y)=(1,0)$. The contraction of a $k$-form $\varpi$ with respect to a vector field $V$ is the $(k-1)$-form defined by  
$\iota_V\varpi_x(v_1,\ldots,v_{k-1})=\varpi_x(V(x),v_1,\ldots,v_{k-1}).$  We conclude that
$$\chi(\omega)=\int_0^1 \eta_t^*(\iota_{\frac{\partial}{\partial t}}\varphi^*\omega)dt.$$

Observe that the family of $k$-forms $\{\eta_t^*(\iota_{\frac{\partial}{\partial t}}\varphi^*\omega)\}_{t\in [0,1]}$ satisfies the hypothesis of Lemma \ref{Leibniz}, then $\chi(\omega)$ is smooth. By definition and the claim it satisfies equality (\ref{dos}). Observe that if $\omega$ is closed, then $\chi(\omega)$ is a primitive of $\omega$, so it is enough to prove the classic Poincaré's lemma. However, in our case we need a primitive in $L^\phi$, so we take a convenient average. 

Define
$$h(\omega)=\frac{1}{\Vol\left(\frac{1}{2}B\right)}\int_{\frac{1}{2}B}\chi_x(\omega)dx,$$
where $\frac{1}{2}B=B\left(0,\frac{1}{2}\right)$.

Since $(x,y)\mapsto \chi_x(\omega)_y$ is smooth in both variables we can use again Lemma \ref{Leibniz} to show that $h(\omega)$ is in $\Omega^k(B)$. Note that this works because we take the integral on a ball with closure included in $B$. Moreover, the derivative of $h$ is
\begin{equation*}
dh(\omega)=\frac{1}{\Vol\left(\frac{1}{2}B\right)}\int_{\frac{1}{2}B}d\chi_x(\omega)dx.
\end{equation*}
Then using (\ref{dos}) we have
\begin{equation}\label{homot}
dh(\omega)+h(d\omega)=\omega
\end{equation}
for all $\omega\in\Omega^k(B)$ with $k\geq 1$.

%As in the $L^p$-version, for $x\in B$ and $\omega\in L^\phi\Omega^k(B)$ we consider the map 
%$$\chi_x(\omega)=\int_0^1 \eta_t^*\left(\iota_{\frac{\partial}{\partial t}}\varphi_x^*\omega\right)dt,$$
%where $\eta_t:B\to [0,1]\times B$, is defined by $\eta_t(y)=(t,y)$, and $\varphi_x:[0,1]\times B$, $\varphi_x(t,y)=ty+(1-t)x$. If $\omega$ is a $k$-form in $L^\phi\Omega^k(B)$ with $k\geq 1$, we put
%$$h(\omega)=\frac{1}{\Vol(\frac{1}{2}B)}\int_{\frac{1}{2}B}\chi_x(\omega)dx,$$
%and
%$$h(f)=\frac{1}{\Vol(\frac{1}{2}B)}\int_{\frac{1}{2}B} f(x)dx,$$
%if $f$ is a function in $L^\phi\Omega^0(B)$. As in Lemma \ref{lemaLpEqAciclico}, we have $h\circ d+d\circ h=\mathrm{Id}$. 

We have to prove that $h$ is well-defined and continuous from $L^\phi\Omega^k(B)$ to $L^\phi\Omega^{k-1}(B)$. To this end we first bound $|\chi_x(\omega)|_y$ for $y\in B$ and $\omega\in\Omega^k(B)$. Since $\iota_{\frac{\partial}{\partial t}}\varphi^*\omega$ is a form on $[0,1]\times B$ that is zero in the direction of $\frac{\partial}{\partial t}$, we have 
$|\eta_t^*(\iota_{\frac{\partial}{\partial t}}\varphi^*\omega)|_y=|\iota_{\frac{\partial}{\partial t}}\varphi^*\omega|_{(t,y)}$ for every $t\in (0,1)$ and $y\in B$. After a direct calculation we get the estimation $|\iota_{\frac{\partial}{\partial t}}\varphi^*\omega|_{(t,y)}\leq t^{k-1}|y-x||\omega|_{\varphi(t,y)}$. Using the assumption that $t\in (0,1)$, we can write
\begin{equation}\label{uno}
|\chi(\omega)|_y\leq \int_0^1|y-x||\omega|_{\varphi(t,y)}dt.
\end{equation}

Consider the function $u:\R^n\to\R$ defined by $u(z)=|\omega|_z$ if $z\in B$ and $u(z)=0$ in the other case. Using (\ref{uno}) and the change of variables $z=ty+(1-t)x$  we have
\begin{align*}
\Vol\left(\mathsmaller{\frac{1}{2}}B\right)|h(\omega)|_y		
&\leq \int_{B\left(ty,\frac{1-t}{2}\right)}\int_0^1|z-y|u(z)(1-t)^{-n-1}dtdz \\
%&\leq \int_{B(y,2)}\int_0^1 \mathds{1}_{B(ty,1-t)}(z)|z-y|u(z)(1-t)^{-n-1}dtdz\\
&=  \int_{B(y,2)}|z-y|u(z)\left(\int_0^1 \mathds{1}_{B\left(ty,\frac{1-t}{2}\right)}(z)(1-t)^{-n-1} dt\right)dz. 	
\end{align*}
Observe that $\mathds{1}_{B\left(ty,\frac{1-t}{2}\right)}(z)=1$ implies that $|z-y|\leq 2(1-t)$. Then we have
$$\int_0^1\mathds{1}_{B\left(ty,\frac{1-t}{2}\right)}(z)(1-t)^{-n-1}dt\leq \int_0^{1-\frac{1}{2}|z-y|}(1-t)^{-n-1}dt= \int_{\frac{1}{2}|z-y|}^1 r^{-n-1}dr\preceq \frac{1}{|z-y|^n}.$$
This implies
$$Vol(\mathsmaller{\frac{1}{2}}B)|h(\omega)|_y\preceq \int_{B(y,2)}|z-y|^{1-n}u(z)dz.$$

Using this estimate we have
\begin{align*}
\|h(\omega)\|_{L^\phi}
    &=\inf\left\{\gamma>0 : \int_B\phi\left(\frac{|h(\omega)|_y}{\gamma}\right)dy\leq 1\right\}\\
    &\preceq \inf\left\{\gamma>0 : \int_B\phi\left(\int_{B(y,2)}|z-y|^{1-n}\frac{u(z)}{\gamma}dz\right)dy\leq 1\right\}.
\end{align*}
Since $\int_{B(y,2)}|z-y|^{1-n}dz<+\infty$, we can use Jensen's inequality and write
\begin{align*}
\|h(\omega)\|_{L^\phi}
    &\preceq \inf\left\{\gamma>0 : \Vol(B(0,3))\int_B\int_{B(0,3)}\phi\left(\frac{u(z)}{\Vol(B(0,3))\gamma}\right)\frac{dz}{|z-y|^{n-1}}dy\leq 1\right\}\\
    &=\inf\left\{\gamma>0 : \Vol(B(0,2))\int_{B(0,3)}\phi\left(\frac{u(z)}{\Vol(B(0,3))\gamma}\right)\left(\int_B\frac{dy}{|z-y|^{n-1}}\right)dz\leq 1\right\}.
\end{align*}
We have that there exists a constant $K>0$ such that $\int_B\frac{dy}{|z-y|^{n-1}}\leq K$ for all $z\in B(0,3)$, thus
$$\|h(\omega)\|_{L^\phi}\preceq \Vol(B(0,2))\|\omega\|_{L^{\tilde{K}\phi}}\preceq \|\omega\|_{L^\phi},$$
where $\tilde{K}=K\Vol(B(0,3))$. 

By the identity $dh(\omega)=\omega-h(d\omega)$ we also have 
$$\|dh(\omega)\|_{\phi}\leq \|\omega\|_{\phi}+\|h(d\omega)\|_{\phi}\preceq \|\omega\|_{\phi}+\|d\omega\|_{\phi}.$$
We conclude that $h$ is well-defined and bounded for the norm $|\ \ |_{L^\phi}$ in all degrees $k\geq 1$.

If $\omega=df$ for certain function $f$ we observe that
$$\eta_t^*(\iota_{\frac{\partial}{\partial t}}\varphi^*_x df) (y)= df_{\varphi_x(t,y)}(y-x)=(f\circ\alpha)'(t),$$
where $\alpha$ is the curve $\alpha(t)=\varphi_x(t,y)$. Then $\chi_x(df)(y)=f(y)-f(x)$, from which we get
$$h(df)=f-\frac{1}{Vol\left(\mathsmaller{\frac{1}{2}}B\right)}\int_{\frac{1}{2}B} f.$$
We define $h:L^p\Omega^0(B)\to L^p\Omega^{-1}(B)=\R$ by
$h(f)=\frac{1}{Vol\left(\mathsmaller{\frac{1}{2}}B\right)}\int_{\frac{1}{2}B} f$,
which is clearly continuous because $\frac{1}{2}B$ has finite Lebesgue measure. Then the identity (\ref{homot}) is true for all $\omega\in L^\phi\Omega^k(B)$ and $h$ is continuous in all degrees.
\end{proof}

The following lemma can be proved by a direct application of the Lipschitz condition.

\begin{lemma}\label{lemaOrliczLipschitz}
Let $M$ and $N$ be two Riemannian manifolds and $f:M\to N$ a bi-Lipschitz diffeomorphism with constant $L$. Then for every $k\in\N$ the pull-back $f^*:L^\phi\Omega^k(N)\to L^\phi\Omega^k(M)$ is continuous and its operator norm is bounded depending on $L$, $k$, $\phi$ and $n=\dim(M)$.
\end{lemma}

A \textit{bicomplex} is a family of topological vector spaces $\{C^{k,\ell}\}_{k,\ell\in\N^2}$ together with continuous linear maps $d':C^{k,\ell}\to C^{k+1,\ell}$ and $d'':C^{k,\ell}\to C^{k,\ell+1}$. We denote it by $(C^{*,*},d',d'')$.

\begin{lemma}[Lemma 5,\cite{Pa95}]\label{weyl} 
Let $(C^{*,*}, d', d'')$ be a bicomplex with 
$d'\circ d'' +d''\circ d' =0$. Suppose that for every $\ell \in \mathbb N$
the complex $(C^{*, \ell}, d')$ retracts to the subcomplex $(E^\ell := \mathrm{Ker}\ d'\vert _{C^{0,\ell}}\to 0 \to 0 \to \cdots)$. Then the complex $(D ^*, \delta)$, defined
by 
$$D^m = \bigoplus _{k + \ell = m} C^{k, \ell}\text{ and }\delta = d' + d'',$$ 
is homotopically equivalent to $(E^*, d'')$.
\end{lemma}

A complete proof of Lemma \ref{weyl} can be found in \cite{S}.

\begin{proof}[Proof of Proposition \ref{PropEq1}]
Let us define the bicomplex
$$C^{k,\ell}=\left\{\omega\in \prod_{U\in\mathcal{U}_\ell}L^\phi\Omega(U):\{\|\omega_U\|_{L^\phi}\}_{U\in\mathcal{U}_\ell},\{\|d\omega_U\|_{L^\phi}\}_{U\in\mathcal{U}_\ell}\in \ell^\phi(\mathcal{U}_\ell)\right\},$$
equipped with the norm 
$$\|\omega\|=\|\theta\|_{\ell^\phi}+\|\theta'\|_{\ell^\phi},$$
where $\theta(U)=\|\omega_U\|_{L^\phi}$ and $\theta'(U)=\|d\omega_U\|_{L^\phi}$. The derivatives are defined by
\begin{itemize}
    \item $(d'\omega)_U=(-1)^\ell d\omega_U$ for every $\omega\in C^{k,\ell}$,
    \item If $\omega\in C^{k,\ell}$ and $W\in\mathcal{U}_{\ell+1}$, $W=U_0\cap\ldots\cap U_{\ell+1}$, then 
    $$(d''\omega)_W=\sum_{i=0}^{\ell+1}(-1)^i(\omega_{U_0\cap\ldots U_{i-1}\cap U_{i+1}\cap\ldots\cap U_{\ell+1}})|_W.$$
\end{itemize}
It is easy to see that $d'$ and $d''$ are well-defined and continuous, and that $d'\circ d''+d''\circ d'=0$.

Observe that the elements of $\mathrm{Ker}\ d'|_{C^{0,\ell}}$ are the functions $g\in\prod_{U\in\mathcal{U}_\ell}L^{\phi}\Omega^0(U)$ satisfying the following conditions:
\begin{itemize}
\item $dg_U=0$ for all $U\in\mathcal{U}_\ell$, then $g_U$ is constant.
\item $\{\|g_U\|_{L^\phi}\}_{U\in\mathcal{U}_\ell}\in \ell^\phi(\mathcal{U}_\ell)=\ell^\phi(X_G^\ell)$.
\end{itemize} 
Using the construction of $X_G$ and the fact that $U$ is bi-Lipschitz (with uniform Lipschitz constant) to the Euclidean unit ball we have that $\mathrm{Ker}\ d'|_{C^{0,\ell}}$ is isomorphic to $\ell^\phi(X_G^\ell)$ and $d''$ coincides with the derivative on this space.

On the other hand, the elements of $\mathrm{Ker}\ d''|_{C^{k,0}}$ are of the form $\omega=\{\omega_U\}_{U\in\mathcal{U}}$ with 
$$\omega_U|_{U\cap U'}=\omega_{U'}|_{U\cap U'}\text{ if } U\cap U'\neq\emptyset.$$ 
We can take a $k$-form $\tilde{\omega}$ in $L^{\phi}\Omega^k(M)$ such that $\tilde{\omega}|_U=\omega_U$ for all $U\in\mathcal{U}$, then there is an isomorphism between $\mathrm{Ker}\ d''|_{C^{k,0}}$ and $L^{\phi} \Omega^k(G)$ for which $d'$ coincides with the derivative on the second space.\\

\underline{Claim 1}: The cochain complex $(C^{*,\ell},d')$ retracts to $(\Ker d'|_{C^{0,\ell}}\to 0\to\ldots)$ for all $\ell\in\N$.

For every $U\in\mathcal{U}_\ell$ consider $f_U:U\to B$ an $L$-bi-Lipschitz diffeomorphism ($L$ does not depend on $U$ and $B$ is the unit ball in the corresponding Euclidean space). We define $H:C^{k,\ell}\to C^{k-1,\ell}$ by
$$(H\omega)_U=(-1)^\ell f_U^* h (f_U^{-1})^*\omega_U,$$
where $h:L^\phi\Omega^k(B)\to L^\phi\Omega^{k-1}(B)$ is the map given by Lemma \ref{alciclicoOrlicz}. Here we consider $C^{-1,\ell}=\Ker d'|_{C^{0,\ell}}$ and $d':C^{-1,\ell}\to C^{0,\ell}$ the inclusion. One can easily verify that $Hd'+d'H=\mathrm{Id}$. Using Lemma \ref{lemaOrliczLipschitz} we have that $H$ is continuous.\\

\underline{Claim 2}: The cochain complex $(C^{k,*},d'')$ retracts to $(\Ker d''|_{C^{k,0}}\to 0\to\ldots)$ for all $k\in\N$.

We have to construct a family of bounded linear maps $P:C^{k,\ell}\to C^{k,\ell-1}$ ($\ell\geq 0$) such that $P\circ d''+d''\circ P =\mathrm{Id}$, where $C^{k,-1}=\mathrm{Ker}\ d''|_{C^{k,0}}$ and $d'':C^{k,-1}\to C^{k,0}$ is the inclusion.

Consider $\{\eta_U\}_{U\in\mathcal{U}}$ a partition of unity with respect to $\mathcal{U}$. If $\ell\geq 1$ and $\omega\in C^{k,\ell}$, then we define
$$(P \omega)_V=\sum_{U\in\mathcal{U}}\eta_U\omega_{U\cap V},$$
for all $V\in\mathcal{U}_{\ell-1}$. For $\omega\in C^{k,0}$ and $V\in\mathcal{U}$ we put
$$(P\omega)_V=\sum_{U\in\mathcal{U}}\eta_U\omega_U|_V.$$
A direct calculation shows that $P$ is as we wanted.\\

Applying Lemma \ref{weyl} we obtain that a complex $(D^*,\delta)$ is homotopically equivalent to $(\mathrm{Ker}\ d'|_{C^{0,*}},d'')$ and $(\mathrm{Ker}\ d''|_{C^{*,0}},d')$. The proof in the non-relative case ends using the above identifications.

To prove the relative case we have to consider the bicomplex $(C^{*,*}_\xi,d',d'')$, where $C^{k,\ell}_\xi$ is the subspace consisting of the elements $\omega$ of $C^{k,\ell}$ for which there exists $V\subset \overline{G}$, a neighborhood of $\xi$, such that $\omega_U\equiv 0$ if $U\subset V$. The above argument works in this case because all maps preserve the subspaces $C^{k,\ell}_\xi$.
\end{proof}

Since $\mathcal{L}^p\Omega^k(G,\mathcal{U})=L^p\Omega^k(G)$, the previous proposition finishes the proof of the $L^p$-case.

\begin{proposition}[Second step]\label{PropEq2}
The complexes $(\mathcal{L}^\phi\Omega^*(G,\mathcal{U}),d)$ and $(\mathcal{I}^\phi\Omega^*(G,\mathcal{U}),d)$ are homotopically equivalent. The same is true for complexes $(\mathcal{L}^\phi\Omega^*(G,\mathcal{U},\xi),d)$ and $(\mathcal{I}^\phi\Omega^*(G,\mathcal{U},\xi),d)$. 
\end{proposition}

Combining Propositions \ref{PropEq1} and \ref{PropEq2} we have the following diagram:

\medskip
\medskip
\medskip

\xymatrix{ &&& \ell^\phi C^0(X_G)\ar[r]^\delta \ar@/_/[d]_\sim  & \ell^\phi C^1(X_G)\ar[r]^\delta \ar@/_/[d]_\sim  & \ell^\phi C^2(X_G)\ar[r]^\delta \ar@/_/[d]_\sim  &\cdots\\
&&& \mathcal{L}^\phi\Omega^0(G,\mathcal{U})\ar[r]^d \ar@/_/[u] \ar@/_/[d]_\sim & \mathcal{L}^\phi\Omega^1(G,\mathcal{U}) \ar[r]^d  \ar@/_/[u] \ar@/_/[d]_\sim & \mathcal{L}^\phi\Omega^2(G,\mathcal{U}) \ar[r]^d \ar@/_/[u] \ar@/_/[d]_\sim &\cdots\\
&&& \mathcal{I}^\phi\Omega^0(G,\mathcal{U})\ar[r]^d \ar@/_/[u] & \mathcal{I}^\phi\Omega^1(G,\mathcal{U}) \ar[r]^d \ar@/_/[u] & \mathcal{I}^\phi\Omega^2(G,\mathcal{U}) \ar[r]^d \ar@/_/[u] & \cdots}

\medskip
\medskip

\begin{proof}
Consider the family of maps $*\kappa:\mathcal{L}^\phi\Omega^k(G,\mathcal{U})\to \mathcal{I}^\phi\Omega^k(G,\mathcal{U})$ given by the convolution with a smooth kernel $\kappa$.\\

\underline{Claim 1}: For a fixed $k=0,\ldots,\dim(G)$ the map  $*\kappa:\mathcal{L}^\phi\Omega^k(G,\mathcal{U})\to \mathcal{L}^\phi\Omega^k(G,\mathcal{U})$ is well-defined and continuous.

Let $\gamma>0$ and $U\in\mathcal{U}$, using Lemma \ref{lemaNormaConvolucion} we have
\begin{align*}\label{eq1}
\int_U\phi\left(\frac{|\omega*\kappa|_x}{\gamma}\right)dx
    &\leq \int_U \phi\left(\int_G\frac{C|\omega|_{x\cdot z}}{\gamma}\kappa(z)dz\right)dx\\
    &\leq\int_U\phi\left(\int_{x\cdot \supp(\kappa)}\frac{C|\omega|_y}{\gamma} dy\right)dx\\
    &\leq \int_U\phi\left(\sum_{U'\in\mathcal{N}_U}\left\|\frac{C\omega|_{U'}}{\gamma}\right\|_{L^1}\right)dx,
\end{align*}
where $\mathcal{N}_U=\{U'\in\mathcal{U}: U'\cap (x\cdot \supp(\kappa))\neq\emptyset\text{ for some } x\in U\}$. The bounded geometry implies that there exists $N$ a uniform bound of $\#\mathcal{N}_U$. Using this bound and Jensen's inequality we have
\begin{align*}
\int_U\phi\left(\frac{|\omega*\kappa|_x}{\gamma}\right)dx
    &\leq \frac{\Vol(U)}{\#\mathcal{N}_U}\sum_{U'\in\mathcal{N}_U}\phi\left(\left\|\frac{NC\omega|_{U'}}{\gamma}\right\|_{L^1}\right).
\end{align*}
Let $V$ be a uniform bound for $\Vol(U)$. Because of Lemma \ref{OrliczL1}, there exists a constant $D$ such that if $\beta$ is in $L^\phi\Omega^k(U)$ with $U\in\mathcal{U}$, then $\|\beta\|_{L^1}\leq D\|\beta\|_{L^\phi}$. Therefore
$$\int_U\phi\left(\frac{|\omega*\kappa|_x}{\gamma}\right)dx
    \leq \frac{V}{\#\mathcal{N}_U}\sum_{U'\in\mathcal{N}_U}\phi\left(\left\|\frac{DNC\omega|_{U'}}{\gamma}\right\|_{L^\phi}\right).$$
If $\gamma\geq DNC\|\omega|_{U'}\|_{L^\phi}$ for all $U'\in\mathcal{N}_U$, then 
$$\int_U\phi\left(\frac{|\omega*\kappa|_x}{\gamma}\right)dx\leq V.$$
By Remark \ref{ObsEqivalenciaOrlicz} there exists a constant $\mathcal{C}(V)$ such that
$$\|\omega*\kappa|_U\|_{L^\phi}\leq \mathcal{C}(V)\|\omega*\kappa|_U\|_{L^\frac{\phi}{V}}\leq \mathcal{C}(V)DNC\sum_{U'\in\mathcal{N}_U}\|\omega|_{U'}\|_{L^\phi}.$$
Denote $L=\mathcal{C}(V)DNC$ and take $\gamma>0$,
\begin{align*}
\sum_{U\in\mathcal{U}}\phi\left(\frac{\|\omega*\kappa|_U\|_{L^\phi}}{\gamma}\right) &\leq \sum_{U\in\mathcal{U}}\phi\left(\frac{L}{\gamma}\sum_{U'\in\mathcal{N}_U}\|\omega|_{U'}\|_{L^\phi}\right)\\
&\leq \sum_{U\in\mathcal{U}}\frac{1}{\#\mathcal{N}_U}\sum_{U'\in\mathcal{N}_U}\phi\left(\frac{NL}{\gamma}\|\omega|_{U'}\|_{L^\phi}\right).    
\end{align*}
Let $R>0$ be such that for every $U'\in\mathcal{U}$, $U'\in\mathcal{N}_U$ for at most $R$ open sets $U\in\mathcal{U}$. Then
$$\sum_{U\in\mathcal{U}}\phi\left(\frac{\|\omega*\kappa|_U\|_{L^\phi}}{\gamma}\right)\leq \sum_{U\in\mathcal{U}} R\phi\left(\frac{NL}{\gamma}\|\omega|_{U'}\|_{L^\phi}\right).$$
This means that, if $\theta(U)=\|\omega|_U\|_{L^\phi}$ and $\vartheta(U)=\|\omega*\kappa|_U\|_{L^\phi}$, then
$$\|\vartheta\|_{\ell^\phi}\leq NL \|\theta\|_{\ell^{R\phi}}\preceq \|\theta\|_{\ell^\phi}.$$
Using the same argument with $d(\omega*\kappa)=d\omega*\kappa$ we can conclude that $|\omega*\kappa|_{\mathcal{L}^\phi}\preceq |\omega|_{\mathcal{L}^\phi}$, which finishes the proof of Claim 1.\\

\underline{Claim 2}: Let $k=0,\ldots,\dim(G)$. The map $*\kappa:\mathcal{L}^\phi\Omega^k(G,\mathcal{U})\to  L^\phi\Omega^k(G)$ is well-defined and continuous. 

As above, if $\gamma>0$ we have, using the same arguments as before, 
\begin{align*}
\int_G\phi\left(\frac{|\omega*\kappa|_x}{\gamma}\right)dx
%    &\leq \int_G\phi\left(\int_G\frac{C|\omega|_{xz}}{\gamma}\kappa(z)dz\right)dx\\
    &\leq
    \sum_{U\in\mathcal{U}}\int_U\phi\left(\frac{|\omega*\kappa|_x}{\gamma}\right)dx
    \\
    &\leq \sum_{U\in\mathcal{U}}\frac{\Vol(U)}{\#\mathcal{N}_U}\sum_{U'\in\mathcal{N}_U}\phi\left(\left\|\frac{DNC \omega|_{U'}}{\gamma}\right\|_{L^phi}\right)\\
    &\leq \sum_{U\in\mathcal{U}} VR\phi\left(\left\|\frac{L \omega|_{U}}{\gamma}\right\|_{L^\phi}\right).
\end{align*}
If we use again the notation $\theta(U)=\|\omega|_U\|_{L^\phi}$, we have $\|\omega\|_{L^\phi}\leq L\|\theta\|_{\ell^{VR\phi}}\preceq \|\theta\|_{\ell^\phi}$. Doing the same with the derivative we conclude the Claim 2.\\ 

Claims 1 and 2 imply that $*\kappa:\mathcal{L}^\phi\Omega^k(G,\mathcal{U})\to \mathcal{I}^\phi\Omega^k(G)$ is well-defined and continuous. Furthermore, by Proposition \ref{PropDerConv} we know that $*\kappa$ commutes with the derivative.

We will define a family of continuous maps $h:\mathcal{L}^\phi\Omega^k(G,\mathcal{U})\to\mathcal{L}^\phi\Omega^{k-1}(G,\mathcal{U})$ such that
\begin{equation}\label{homotopiaOrlicz1}
\left\{\begin{array}{cc}
        h(df)=f-i(f*\kappa) &\text{ if }f\in \mathcal{L}^\phi\Omega^0(G,\mathcal{U})\\
        h(d\omega)+dh(\omega)=\omega-i(\omega*\kappa) &\text{ if }\omega\in \mathcal{L}^\phi\Omega^k(G,\mathcal{U}),\ k\geq 1,
\end{array}\right.
\end{equation}
where $i$ denotes the inclusion (which is clearly continuous). If $h$ maps continuously $\mathcal{I}^\phi\Omega^k(G,\mathcal{U})$ on $\mathcal{I}^\phi\Omega^{k-1}(G,\mathcal{U})$ for every $k\geq 1$, the complexes $(\mathcal{L}^\phi\Omega^*(G,\mathcal{U}),d)$ and $(\mathcal{I}^\phi\Omega^*(G,\mathcal{U}),d)$ are homotopically equivalent. 
%\begin{equation*}
%\left\{\begin{array}{cc}
%        h(df)=f-(i f)*\kappa &\text{ if }f\in \mathcal{I}^\phi\Omega^0(G,\mathcal{U})\\
%        h(d\omega)+dh(\omega)=\omega-(i %\omega)*\kappa &\text{ if }\omega\in \mathcal{I}^\phi\Omega^k(G,\mathcal{U}),\ k\geq 1.
%\end{array}\right.
%\end{equation*}

\medskip
\medskip
\medskip

\xymatrix{ &&& \mathcal{L}^\phi\Omega^0(G,\mathcal{U})\ar[r]_d \ar@/_/[d]_{*\kappa} & \mathcal{L}^\phi\Omega^1(G,\mathcal{U}) \ar[r]_d  \ar@/_/[d]_{*\kappa} \ar@/_1pc/[l]_h & \mathcal{L}^\phi\Omega^2(G,\mathcal{U}) \ar[r] \ar@/_/[d]_{*\kappa} \ar@/_1pc/[l]_h &\cdots \\
&&& \mathcal{I}^\phi\Omega^0(G,\mathcal{U})\ar[r]^d \ar@/_/[u]_i & \mathcal{I}^\phi\Omega^1(G,\mathcal{U}) \ar[r]^d \ar@/_/[u]_i \ar@/^1pc/[l]^h & \mathcal{I}^\phi\Omega^2(G,\mathcal{U}) \ar[r] \ar@/_/[u]_i \ar@/^1pc/[l]^h &\cdots}

\medskip
\medskip
\medskip

For every $z\in \supp(\kappa)$ we consider $Z\in \Lie(G)$ a left-invariant vector field such that $\exp(Z)=z$. Let $\varphi_t^Z$ be the flow associated with $Z$. Observe that 
$$\varphi_t^Z(x)=L_x \varphi^Z_t(e)=x\cdot \exp(tZ).$$

Given $\omega\in\mathcal{L}^\phi\Omega^k(G,\mathcal{U})$ with $k=1,\ldots,\dim (G)$, we define
$$h(\omega)_x=-\int_G\left(\int_0^1(\varphi_t^Z)^*\iota_Z\omega_x dt\right)\kappa(z)dz.$$
By Lemma \ref{Leibniz} it is smooth and its derivative is
$$dh(\omega)_x=-\int_G\left(\int_0^1(\varphi_t^Z)^*d\iota_Z\omega_x dt\right)\kappa(z)dz.$$
 
\underline{Claim 3}: $h$ and $*\kappa$ verify \eqref{homotopiaOrlicz1}.

Take $\omega$ a $k$-form in $\mathcal{L}^\phi\Omega^k(G,\mathcal{U})$ for $k\geq 1$, thus
\begin{align*}
h(d\omega)+dh(\omega)
    &=-\int_G\left(\int_0^1(\varphi_t^Z)^*(\iota_Zd+d\iota_Z)\omega\ dt\right)\kappa(z)dz.
\end{align*}

Recall the Cartan formula $L_Z\omega=\iota_Zd\omega+d\iota_Z\omega$ (see for example \cite{GHL}). Using the identity $\frac{\partial}{\partial t}(\varphi_t^Z)^*\omega=(\varphi_t^Z)^*L_Z\omega$ we obtain
\begin{align*}
h(d\omega)+dh(\omega)
    &=-\int_G\left(\int_0^1(\varphi_t^Z)^*L_Z\omega\ dt\right)\kappa(z)dz\\
    &=-\int_G\left(\int_0^1\frac{\partial}{\partial t}(\varphi_t^Z)^*\omega\ dt\right)\kappa(z)dz\\
    &=-\int_G((\varphi_t^Z)^*\omega-\omega)\kappa(z)dz=\omega-\omega*\kappa.
\end{align*}

Now consider $f\in\mathcal{L}^\phi\Omega^0(G,\mathcal{U})$. We have
$$h(df)_x =-\int_G\left(\int_0^1 df_{x\cdot \exp(tZ)}(Z(x))dt\right)\kappa(z)dz.$$
Let $\alpha:[0,1]\to G$ be the curve $\alpha(t)=x\cdot \exp(tZ)$. We have $\alpha'(t)=Z(\alpha(t))$, then
$$(f\circ\alpha)'(t)=d_{\alpha(t)}f(\alpha'(t))=df_{x\cdot \exp(tZ)}(Z(\alpha(t))).$$
Therefore
\begin{align*}
h(df)_x
    &=-\int_G \left(\int_0^1(f\circ\alpha)'(t)dt\right)\kappa(z)dz\\
    &=\int_G(f(\alpha(1))-f(\alpha(0)))\kappa(z)dz\\
    &=f(x)-f*\kappa(x).
\end{align*}

\underline{Claim 4}: $h:\mathcal{L}^\phi\Omega^k(G,\mathcal{U})\to\mathcal{L}^\phi\Omega^{k-1}(G,\mathcal{U})$ is well-defined and  continuous for every $k=1,\ldots,\dim(G)$.

First we estimate the operator norm of $h\omega$ at a point $x\in G$. Consider $v_1,\ldots,v_{k-1}\in T_xG$, then
\begin{align*}
|h(\omega)_x(v_1,&\ldots,v_{k-1})|
    =\left|\int_G\left(\int_0^1\omega_{\varphi_t^Z(x)}(Z(\varphi_t^Z(x)),d_x\varphi_t^Z(v_1),\ldots,d_x\varphi_t^Z(v_{k-1}))dt \right)\kappa(z)dz\right|\\
    &\leq \int_G\left(\int_0^1|\omega_{\varphi_t^Z(x)}(Z(\varphi_t^Z(x)),d_xR_{\exp(tZ)}(v_1),\ldots,d_xR_{\exp(tZ)}(v_{k-1}))|dt \right)\kappa(z)dz
\end{align*}

As in the proof of Lemma \ref{lemaNormaConvolucion} we have a uniform bound $|d_xR_{\exp(tZ)}| \leq M$ for all $z\in \supp(\kappa)$. Moreover, since left-invariant fields have constant norm we can write $\|Z(y)\|_y= C$ for every $y\in G$. Hence, if $\|v_1\|_x=\ldots=\|v_{k-1}\|_x=1$,
$$|h(\omega)_x(v_1,\ldots,v_{k-1})|\leq \int_G \left(\int_0^1 CM^{k-1}|\omega|_{\varphi_t^Z(x)}dt \right)\kappa(z)dz,$$
which implies
\begin{equation}\label{estimacionh}
|h(\omega)|_x\leq \int_G \left(\int_0^1 CM^{k-1}|\omega|_{\varphi_t^Z(x)}dt \right)\kappa(z)dz.
\end{equation}

Using \eqref{estimacionh} and Jensen's inequality we obtain
\begin{equation}\label{phih}
\phi\left(\frac{|h(\omega)|_x}{\gamma}\right)
\leq \int_G\left(\int_0^1\phi\left(\frac{CM^{k-1}}{\gamma}|\omega|_{\varphi_t^Z(x)}\right)dt\right)\kappa(z)dz.
\end{equation}
For $U\in\mathcal{U}$ denote $\theta(U)=\|\omega|_U\|_{L^\phi}$ and  $\vartheta(U)=\|h\omega|_U\|_{L^\phi}$. If $\gamma>0$ we have
\begin{align*}
\int_U\phi\left(\frac{|h(\omega)|_x}{\gamma}\right)dx
    &\leq \int_U\int_G\left(\int_0^1\phi\left(\frac{CM^{k-1}}{\gamma}|\omega|_{\varphi_t^Z(x)}dt\right)\right)\kappa(z)dzdx\\
    &=\int_G\left(\int_0^1\int_U\phi\left(\frac{CM^{k-1}}{\gamma}|\omega|_{\varphi_t^Z(x)}\right)dxdt\right)\kappa(z)dz.
\end{align*}
The identity $d_x\varphi^Z_t=d_xR_{\exp(tZ)}$ allows us to find $m>0$ such that $m<|Jac_x(\varphi_t^Z)|$ for all $z\in \supp(\kappa)$. Then
\begin{align*}
\int_U\phi\left(\frac{|h(\omega)|_x}{\gamma}\right)dx
    &\leq \int_G\left(\int_0^1\int_{E(U)}\frac{1}{m}\phi\left(\frac{CM^{k-1}}{\gamma}|\omega|_y\right)dydt\right)\kappa(z)dz\\
    &=\frac{1}{m}\int_{E(U)}\phi\left(\frac{CM^{k-1}}{\gamma}|\omega|_y\right)dy,
\end{align*}
where $E(U)$ is a neighborhood of $U$ with uniform radius (independent of $U$) such that $\varphi_t^Z(x)\in E(U)$ for all $z\in \supp(\kappa)$ and $x\in U$. Consider
$$\mathcal{V}_U=\{V\in\mathcal{U} : V\cap E(U)\neq\emptyset\},$$
then if $\gamma\geq CM^{k-1}\max\{\|\omega|_V\|_{L^{\frac{S}{m}\phi}} : V\in\mathcal{V}_U\}$, where $S\geq \#\mathcal{V}_U$ for all $U\in\mathcal{U}$,
$$\int_U\phi\left(\frac{|h(\omega)|_x}{\gamma}\right)dx\leq 1,$$
which implies
$$\|h(\omega)|_U\|_{L^\phi}\leq CM^{k-1}\max\{\|\omega|_V\|_{L^{\frac{S}{m}\phi}} : V\in\mathcal{V}_U\}\leq \mathcal{M}\sum_{V\in\mathcal{V}_U}\|\omega|_V\|_{L^\phi}$$
for some constant $\mathcal{M}$ that does not depend on $U$. Therefore
\begin{align*}
\sum_{U\in\mathcal{U}}\phi\left(\frac{\|h(\omega)|_U\|_{L^\phi}}{\gamma}\right)
    &\leq \sum_{U\in\mathcal{U}}\phi\left(\mathcal{M}\sum_{V\in\mathcal{V}_U}\frac{\|\omega|_V\|_{L^\phi}}{\gamma}\right)\\
    &\leq \sum_{U\in\mathcal{U}}\sum_{V\in\mathcal{V}_U}\frac{1}{\#\mathcal{V}_U}\phi\left(\frac{S\mathcal{M}\|\omega|_V\|_{L^\phi}}{\gamma}\right)\\
    &\leq \sum_{U\in\mathcal{U}}\mathcal{N}\phi\left(\frac{S\mathcal{M}\|\omega|_V\|_{L^\phi}}{\gamma}\right),
\end{align*}
where $\mathcal{N}\geq \#\{U\in\mathcal{U} : V\in\mathcal{V}_U\}$ for all $V\in\mathcal{U}$. From here we obtain
$$\|\vartheta\|_{\ell^\phi}\leq S\mathcal{M}\|\theta\|_{\ell^{\mathcal{N}\phi}}\preceq \|\theta\|_{\ell^\phi}.$$

Using the identity $dh(\omega)=\omega-i(\omega*\kappa)$ and the above estimate we obtain $$|h(\omega)|_{\mathcal{L}^\phi}\preceq |\omega|_{\mathcal{L}^\phi}.$$

\underline{Claim 5:} The map $h:L^\phi\Omega^k(G)\to L^\phi\Omega^{k-1}(G)$ is well-defined and continuous for every $k=1,\ldots,\dim(G)$. 

Using \eqref{phih} we have
\begin{align*}
\int_G\phi\left(\frac{|h(\omega)|_x}{\gamma}\right)dx
    &\leq \int_G\int_G\left(\int_0^1\phi\left(\frac{CM^{k-1}|\omega|_{\varphi_t^Z(x)}}{\alpha}\right)dt\right)\kappa(z)dzdx\\
    &\leq \int_G\left(\int_0^1\int_G\frac{1}{m}\phi\left(\frac{CM^{k-1}|\omega|_y}{\gamma}\right)dydt\right)\kappa(z)dz\\
    &=\int_G\frac{1}{m}\phi\left(\frac{CM^{k-1}|\omega|_y}{\gamma}\right)dy.
\end{align*}
From this we obtain $\|h(\omega)\|_{L^\phi}\preceq \|\omega\|_{L^\phi}$; and using again the equality (\ref{homotopiaOrlicz1}) we have $|h(\omega)|_{L^\phi}\preceq |\omega|_{L^\phi}$.

By Claims 4 and 5 we conclude that $h$ is well-defined and continuous from $\mathcal{I}^\phi\Omega^k(G,\mathcal{U})$ to $\mathcal{I}^\phi\Omega^{k-1}(G,\mathcal{U})$, which finishes the poof in the non-relative case.

The same argument works in the relative case. The only thing we have to verify is that the maps $*\kappa$ and $h$ preserve the relative subcomplexes, which is easy using the compactness of $\supp(\kappa)$.
\end{proof}

The proof of Theorem \ref{main} finishes with the following proposition.

\begin{proposition}\label{PropEq3}
The complexes $(\mathcal{I}^\phi\Omega^*(G,\mathcal{U}),d)$, $(L^\phi\Omega^*(G),d)$, and $(L^\phi C^*(G),d)$ are homotopically equivalent. The same result is true for the corresponding relative complexes.
\end{proposition}

\begin{proof}
In this case we consider $*\kappa$ and $h$ defined as in Proposition \ref{PropEq2}. We have to prove that they are well-defined and continuous. Identities as (\ref{homotopiaOrlicz1}) are clearly satisfied.

\medskip
\medskip
\medskip

\xymatrix{ &&& L^\phi C^0(G)\ar[r]^d \ar@/_/[d]_{*\kappa}  & L^\phi C^1(G)\ar[r]^d \ar@/_/[d]_{*\kappa} \ar@/_1pc/[l]_h  & L^\phi C^2(G)\ar[r]^d \ar@/_/[d]_{*\kappa} \ar@/_1pc/[l]_h  &\cdots\\
&&& \mathcal{I}^\phi\Omega^0(G,\mathcal{U})\ar[r]^d \ar@/_/[u]_i \ar@/_/[d]_i & \mathcal{I}^\phi\Omega^1(G,\mathcal{U}) \ar[r]^d  \ar@/_/[u]_i \ar@/_/[d]_i \ar@/_1pc/[l]_h & \mathcal{I}^\phi\Omega^2(G,\mathcal{U}) \ar[r]^d \ar@/_/[u]_i \ar@/_/[d]_i \ar@/_1pc/[l]_h &\cdots\\
&&& L^\phi\Omega^0(G)\ar[r]^d \ar@/_/[u]_{*\kappa} & L^\phi\Omega^1(G) \ar[r]^d \ar@/_/[u]_{*\kappa} \ar@/_1pc/[l]_h & L^\phi\Omega^2(G) \ar[r]^d \ar@/_/[u]_{*\kappa} \ar@/_1pc/[l]_h &\cdots}

\medskip
\medskip
\medskip

The map $h:L^\phi C^k(G)\to L^\phi C^{k-1}(G)$ is the continuous extension of $h:L^\phi \Omega^k(G)\to L^\phi \Omega^{k-1}(G)$, then we have that all maps $h$ in the diagram are continuous.

We have to prove that the map $*\kappa:L^\phi C^k(G)\to\mathcal{I}^\phi\Omega^k(G,\mathcal{U})$ is well-defined and continuous for every $k=0,\ldots,\dim(G)$. Then so is $*\kappa:L^\phi\Omega^k(G)\to\mathcal{I}^\phi\Omega^k(G,\mathcal{U}).$ To this end, first observe that if $\omega\in L^\phi C^k(G)$ then $\omega*\kappa\in\Omega^k(G)$ by Lemma \ref{PropDerConv}.

Let $\gamma>0$. Using the estimate given in Lemma \ref{lemaNormaConvolucion} we have
\begin{align*}
\int_G\phi\left(\frac{|\omega*\kappa|_x}{\gamma}\right)dx
    &\leq \int_G\phi\left(\frac{C|\omega|*\kappa(x)}{\gamma}\right)dx\\
    &=\int_G\phi\left(\int_G\frac{C|\omega|_{xz}}{\gamma}\kappa(z)dz\right)dx\\
    &\leq \int_G\int_G\phi\left(\frac{C|\omega|_{xz}}{\gamma}\right)\kappa(z)dzdx.
\end{align*}
In the last line we use Jensen's inequality. As before we take $m>0$ with $m<|Jac_x(R_z)|$ for all $x\in G$ and $z\in \supp(\kappa)$, then
\begin{align*}
\int_G\phi\left(\frac{|\omega*\kappa|_x}{\gamma}\right)dx
    &= \int_G\left(\int_G\phi\left(\frac{C|\omega|_{xz}}{\gamma}\right)\frac{|Jac_x(R_z)|}{|Jac_x(R_z)|}dx\right)\kappa(z)dz\\
    &\leq \int_G\left(\int_G\frac{1}{m}\phi\left(\frac{C|\omega|_y}{\gamma}\right)dy\right)\kappa(z)dz\\
    &=\int_G\frac{1}{m}\phi\left(\frac{C|\omega|_y}{\gamma}\right)dy.
\end{align*}
If $\gamma\geq C\|\omega\|_{L^{\frac{\phi}{m}}}$ we have 
$$\int_G\phi\left(\frac{|\omega*\kappa|_x}{\gamma}\right)dx\leq 1,$$
which implies $\|\omega*\kappa\|_{L^\phi}\leq C\|\omega\|_{L^{\frac{\phi}{m}}}\preceq \|\omega\|_{L^\phi}$. In the same way we have $\|d(\omega*\kappa)\|_{L^\phi}=\|d\omega*\kappa\|_{L^\phi}\preceq \|d\omega\|_{L^\phi}$ and as a conclusion $|\omega*\kappa|_{L^\phi}\preceq |\omega|_{L^\phi}$.

On the other hand we denote $\vartheta(U)=\|\omega*\kappa|_U\|_{L^\phi}$ and estimate
\begin{align*}
\int_U\phi\left(\frac{|\omega*\kappa|_x}{\gamma}\right)dx
    &\leq \int_U\phi\left(\int_G \frac{C|\omega|_{xz}}{\gamma}\kappa(z)dz\right)dx\\
    &\leq D\phi\left(\int_{E(U)} \frac{C|\omega|_y}{\gamma}dy\right),
\end{align*}
where $E(U)$ is a neighborhood of $U$ with radius independent of $U$, and $D$ is a constant (also independent of $U$). We can deduce from this that
$$\|\omega*\kappa|_U\|_{L^\phi}\leq \frac{C}{\phi^{-1}(1/D)}\int_{E(U)}|\omega|_ydy.$$
In order to simplify the notation we write $\mathcal{C}=\frac{C}{\phi^{-1}(1/D)}$. Then
\begin{align*}
\sum_{U\in\mathcal{U}}\phi\left(\frac{\|\omega*\kappa|_U\|_{L^\phi}}{\gamma}\right)
    &\leq \sum_{U\in\mathcal{U}} \phi\left(\frac{\mathcal{C}}{\gamma}\int_{E(U)}|\omega|_y dy\right)\\
    &\leq \sum_{U\in\mathcal{U}}\frac{1}{\Vol(E(U))}\int_{E(U)}\phi\left(\frac{\mathcal{C}\Vol(E(U))|\omega|_y}{\gamma}\right)dy.
\end{align*}
Using that $\{E(U): U\in\mathcal{U}\}$ is a uniformly locally finite covering such that $\Vol(E(U))$ is bounded from above and below far from zero, we can find a uniform constant $L$ such that 
$$\sum_{U\in\mathcal{U}}\frac{1}{\Vol(E(U))}\int_{E(U)}\phi\left(\frac{\mathcal{C}\Vol(E(U))|\omega|_y}{\gamma}\right)dy\leq \int_G L\phi\left(\frac{L|\omega|_y}{\gamma}\right)dy.$$
This proves that $\|\vartheta\|_{\ell^\phi}\preceq \|\omega\|_{L^\phi}$. Doing the same for the derivative we obtain $|\omega*\kappa|_{\mathcal{L}^\phi}\preceq |\omega|_{L^\phi}$, that finish de proof of the Claim.

As in Proposition \ref{PropEq2} the relative case follows from the previous argument. 
\end{proof}

%%%%%%%%%%%%%%

\bibliographystyle{alpha}
\bibliography{ref.bib}

\subsection*{Acknowledgments} 
This work is part of my thesis, which was supervised by Marc Bourdon and Matías Carrasco. I thank them for their guidance and valuable ideas.

\medskip
\medskip

\Addresses

\end{document}